\documentclass[12pt,a4paper]{amsart}
\usepackage[left=2.4cm,right=2.4cm,top=2.8cm,bottom=2.6cm, footskip=18pt]{geometry}
\usepackage{amsmath,amssymb,amsfonts,amscd,amsthm,wasysym,xcolor,enumitem,indentfirst,graphicx,booktabs,mathtools}
\usepackage[bookmarksnumbered,colorlinks,linktocpage,pagebackref,plainpages]{hyperref}
\hypersetup{pdfstartview={FitH}}

\newtheorem{thmx}{Theorem}

\numberwithin{equation}{section}
\newtheorem{theorem}{Theorem}[section]

\newtheorem{lemma}[theorem]{Lemma}

\theoremstyle{definition}
\newtheorem{definition}[theorem]{Definition}

\newtheorem{question}[theorem]{Question}
\newtheorem{remark}[theorem]{Remark}

\newtheorem{Qx}[thmx]{Question}

\allowdisplaybreaks[4]

\makeatletter
\@namedef{subjclassname@2020}{\textup{2020} Mathematics Subject Classification}
\makeatother

\begin{document}
\title[$L^2$ extension and removable singularities of psh functions]{On the Ohsawa--Takegoshi $L^2$ extension theorem and removable singularities of plurisubharmonic functions}
\date{\today}

\author{Xieping Wang}
\address{CAS Wu Wen-Tsun Key Laboratory of Mathematics and School of Mathematical Sciences, University of Science and Technology of China, Hefei 230026, Anhui, People's Republic of China}
\email{xpwang008@ustc.edu.cn}

\thanks{The author was partially supported by the NSFC (Grant No. 12371083) and the Fundamental Research Funds for the Central Universities (Grant Nos. WK0010000099 and WK3470000030).}


\dedicatory{To Li and Rui'an}

\begin{abstract}
The celebrated Ohsawa--Takegoshi extension theorem for $L^2$ holomorphic functions on bounded pseudoconvex domains in $\mathbb C^n$ is a fundamental result in the fields of complex analysis and algebraic geometry. In 1995, Ohsawa conjectured that the theorem holds more generally on  bounded complete K\"ahler domains in $\mathbb C^n$. Recently, Chen, Wu and Wang confirmed this conjecture in a special case. In this paper, we extend their result to the case of holomorphic sections of twisted canonical bundles over relatively compact complete K\"{a}hler domains in Stein manifolds. As an application, we establish a Hartogs-type extension theorem for plurisubharmonic functions across compact complete pluripolar sets. This result complements a classical theorem of Shiffman and may be regarded as a plurisubharmonic analogue of the Skoda--El Mir extension theorem, thereby filling a gap that appears to have remained open in the literature since at least 1985.
\end{abstract}
\maketitle

\section{Introduction}
One of the central results in the $L^2$ theory of several complex variables is the Ohsawa--Takegoshi $L^2$ extension theorem, a classical version of which can be stated as follows:

\begin{thmx}[see \cite{OhsawaTakegoshi87}]\label{thm:original OT-extension}
Let $\Omega$ be a bounded pseudoconvex domain in $\mathbb C^n$ and $H$ be a complex affine subspace of $\mathbb C^n$ intersecting $\Omega$. Then for every plurisubharmonic function $\varphi$ on $\Omega$ and every holomorphic function $f$ on $\Omega\cap H$ satisfying
  $$
  \int_{\Omega\cap H}|f|^2e^{-\varphi}\, {\rm d}V_H<\infty,
  $$
there exists a holomorphic function $F$ on $\Omega$ such that $\left.F\right|_{\Omega\cap H}=f$ and
  $$
  \int_{\Omega}|F|^2e^{-\varphi}\, {\rm d}V\leq C\int_{\Omega\cap H}|f|^2e^{-\varphi}\, {\rm d}V_H,
  $$
where $C>0$ is a constant depending only on the diameter of $\Omega$ and the codimension of $H$ in $\mathbb C^n$, and ${\rm d}V$ {\rm(}resp. ${\rm d}V_H${\rm)} denotes the Lebesgue measure on $\mathbb C^n$ {\rm(}resp. $H${\rm)}.
\end{thmx}

Since the groundbreaking work of Ohsawa and Takegoshi \cite{OhsawaTakegoshi87}, an extensive literature has developed around the above theorem, encompassing its far-reaching generalizations and applications in various contexts. The Ohsawa--Takegoshi $L^2$ extension theorem has proved to be of fundamental importance in the fields of several complex variables and complex algebraic geometry,  serving as a highly effective tool for producing, by induction on dimension, holomorphic functions (or, more generally, holomorphic sections) with uniformly controlled $L^2$ norms. We refer the reader to \cite{Berndtsson-note, Demaillybook2, Ohsawabook, Siu_Extension02} for comprehensive treatments of this subject, and also to the surveys  \cite{Blocki_Survey14, McN-Var_survey15, Berndtsson_ICM18, Ohsawa_L2-survey} for brief introductions. Significant recent advances include work by B{\l}ocki \cite{Blocki13} and by Guan and Zhou \cite{GZ_optimal, GZ_openness}. In particular, Guan and Zhou \cite{GZ_optimal} established an optimal-constant  version of the Ohsawa--Takegoshi $L^2$ extension theorem in a very general setting, with many impressive applications. Their work provided further impetus to research in this area;  see, for example, \cite{Berndtsson-Lempert, Cao-Paun-Bern24, Chan-Choi22, DNWZ23_Math-Ann, DWWZ_AJM, GMY25_optimal, GMY26_optimal, Hosono_jet-ext, Kim_JMPA, LXZ_Math-Z, LX_JFA, Rao-Zhang_jet-ext, XZ_Openness-Extension, ZZ_optimal, ZZ_MathAnn}.

Despite the above achievements, several questions about the $L^2$ extension of holomorphic functions remain open,  including the following long-standing one posed by Ohsawa himself (compare the first paragraph of \cite[Sect. 6.2.1]{Berndtsson-note}):

\begin{Qx}[see \cite{Ohsawa95}]\label{Quest:Ohsawa}\footnote{Although formulated slightly differently, this question is actually equivalent to Ohsawa's original question, as can be easily seen by induction on $n$.}
Does Theorem \ref{thm:original OT-extension} hold more generally for bounded {\it complete K\"ahler} domains in $\mathbb C^n$?
\end{Qx}

For background on  and motivation for this question, we refer the reader to Ohsawa's original paper \cite{Ohsawa95} (and his recent book \cite{Ohsawabook}). We only note  here that, as has been amply demonstrated by the work of Demailly \cite{Demailly82},  H\"{o}rmander's weighted $L^2$ theory for the $\bar{\partial}$-operator on Stein manifolds extends equally well to more general complete K\"{a}hler manifolds; it is therefore natural to expect  that the $L^2$ extension theory of Ohsawa--Takegoshi, which builds on the former, does as well. At present, however, Question \ref{Quest:Ohsawa} appears difficult to resolve in full generality. To the best of our knowledge, the only result known to date is due to Chen, Wu and Wang \cite{CWW}, who gave an affirmative answer in the special case in which $H$ is a single point. Their result  has found a few unexpected applications; see \cite{CWW} for details.

In this paper we first extend the above result of Chen, Wu and Wang to the case of holomorphic sections of twisted canonical bundles over relatively compact complete K\"{a}hler domains in Stein manifolds, and then present some new applications. Our first main result reads as follows.

\begin{theorem}\label{thm:OT-extension-twisted}
Let $X$ be an $n$-dimensional Stein manifold with a continuous volume form ${\rm d}V_X$ and $\Omega$ be a relatively compact complete K\"ahler domain in $X$. Suppose $L$ is a holomorphic line bundle over $\Omega$ and $h$ is a singular Hermitian metric on $L$ with semi-positive curvature current. Then for every $x_0\in \Omega$ and $s_{x_0}\in \Lambda^n T^{\ast}_{x_0}\Omega \otimes L_{x_0}$, there is an $L$-valued holomorphic $n$-form $s$ on $\Omega$ such that $s(x_0)=s_{x_0}$ and
   $$
   \int_{\Omega} \sqrt{-1}^{\,n^2}\{s, s\}_h \leq C\frac{\sqrt{-1}^{\,n^2}\{s_{x_0}, s_{x_0}\}_h}{{\rm d}V_X(x_0)}
   $$
provided the right-hand side is finite, where $C>0$ is a constant depending only on $n$, $\Omega$ and ${\rm d}V_X$.
\end{theorem}

Here $\{\,\, ,\,\}_h$ denotes the sesquilinear product combining the wedge product of scalar-valued forms with the Hermitian metric $h$ on $L$ (see
Section \ref{sect:Ohsawa--Takegoshi extension} below for a precise definition).

We now sketch the idea of the proof of Theorem \ref{thm:OT-extension-twisted}. Following the strategy of Ohsawa \cite{Ohsawa95} and Chen, Wu and Wang \cite{CWW}, we reduce Theorem \ref{thm:OT-extension-twisted} to an Ohsawa--Takegoshi-type $L^2$ existence theorem for the $\bar{\partial}$-operator on Hermitian holomorphic line bundles, subject to certain curvature constraints, over \textit{complete K\"ahler} manifolds (see Theorem \ref{thm:L^2 estimate of OT-type} below). Since the Hermitian metric on the line bundle under consideration is allowed to be \textit{singular}, the proof of Theorem \ref{thm:L^2 estimate of OT-type} relies heavily on Demailly's approximation theory \cite{Demailly82, Demailly92,Demailly94} (see also \cite{BK07_psh-app}), in which there are mainly two approximation theorems for singular Hermitian metrics:
\begin{itemize}[leftmargin=2.0pc, parsep=2pt]
\item  one is a local approximation by Hermitian metrics that are smooth on the complement of suitably chosen complex subvarieties in the manifold, with arbitrarily small loss of positivity in curvature;

\item  the other is a global approximation by Hermitian metrics that are smooth on the entire manifold, but with less controllable loss of positivity in curvature.
\end{itemize}
The first one has been used by Zhou and Zhu \cite{ZZ_optimal} (see also \cite{ZZ_MathAnn}) to establish a singular-metric  (and optimal-constant) version of the Ohsawa--Takegoshi--Manivel $L^2$ extension theorem on weakly pseudoconvex K\"{a}hler manifolds \cite{Demailly-Manivel_extension}. In the more general case of complete K\"{a}hler manifolds, the lack of pseudoconvexity however makes it hard to apply the first approximation theorem. Instead we make use of the second one, which combined with the a priori inequality for the twisted $\bar{\partial}$-operator due to Ohsawa--Takegoshi \cite{OhsawaTakegoshi87} enables us to solve the $\bar{\partial}$-equation with two error terms via an induction argument, ultimately leading to Theorem \ref{thm:L^2 estimate of OT-type}. This idea of arguing by induction, inspired by the work of
Demailly \cite{Demailly82}, differs from that of Chen, Wu and Wang \cite{CWW} even when applied to the case they considered (see Remark \ref{rmk:supplemental remark}(iv) for further explanation).

Next we present an application of Theorem \ref{thm:OT-extension-twisted} to the study of removable singularities of plurisubharmonic (psh for short) functions. Specifically, we prove the following somewhat surprising result.

\begin{theorem}\label{thm:Hartogs_PSH-manifold}
Let $\Omega$ be a domain in a Stein manifold of dimension $\geq 2$ and $K$ be a compact complete pluripolar subset of $\Omega$. Then every psh function on $\Omega\!\setminus\!K$ admits a unique psh extension to $\Omega$.
\end{theorem}

 It is worth emphasizing here that the full strength of Theorem \ref{thm:OT-extension-twisted} is required for the proof of this theorem, while its usual version in which $\Omega$ is Stein is not sufficient, as will be seen in Section \ref{sect: psh extension}.

Recently, Chen \cite{Chen_Har-Sob} showed that the well-known Hartogs extension theorem for holomorphic functions is also valid for pluriharmonic functions (see also \cite{Wang_Pisa}). However, it has long been known that this is not the case for psh functions (see, e.g., \cite{Bed-Tay_subext}). Given this situation,  it is worthwhile to  establish a Hartogs extension theorem of a slightly different nature for psh functions, as stated above.  Our motivation for this result also stems from  the renewed interest  in the original Hartogs extension theorem in recent years; for details see the most recent work \cite{BDP22_Hartogs, Vijiitu} and the references therein.

To put Theorem \ref{thm:Hartogs_PSH-manifold} in perspective, we recall that there are two landmark results in the extension theory of closed positive currents, namely the Harvey extension theorem \cite{Harvey_Extension} and the Skoda--El Mir extension theorem \cite{Skoda_Extension82, El-Mir_cmp-plolar}. These two theorems complement each other, and each unifies and generalizes many important earlier results due to Bishop, Remmert--Stein, Shiffman, Siu and others; for more information see the original papers just mentioned and also \cite{Harvey_CPAM, Harvey_Survey, Sibony85}. Corresponding to Harvey's extension theorem there is an extension theorem for psh functions due to Shiffman \cite{Shiffman72}, which states that every psh function on an $n$-dimensional complex manifold extends plurisubharmonically across a closed set of Hausdorff $(2n-2)$-measure zero. Theorem \ref{thm:Hartogs_PSH-manifold} may be regarded, in a broad sense, as a plurisubharmonic analogue of the Skoda--El Mir extension theorem. To the best of our knowledge, no such analogue has previously appeared in the literature.

To further clarify the scope of Theorem \ref{thm:Hartogs_PSH-manifold}, we recall the standard fact from pluripotential theory that pluripolar sets in an $n$-dimensional complex manifold have Hausdorff dimension at most $2n-2$. We also note that there exist compact complete pluripolar sets in $\mathbb C^n$ whose Hausdorff $(2n-2)$-measure is positive and may be either finite or infinite; see Subsection \ref{sect:Complete pluripolar sets} for precise references. Together, these facts show that Theorem \ref{thm:Hartogs_PSH-manifold} is complementary to Shiffman's theorem described above, in the sense that each covers cases not covered by the other. Furthermore, Theorem \ref{thm:Hartogs_PSH-manifold} is  global in nature, whereas  Shiffman's theorem, like many other related results in the literature (see, for example, \cite{Siu74_Analyticity, Harvey_CPAM, CWW}),  is local in nature.

Interestingly, the proof of Theorem \ref{thm:Hartogs_PSH-manifold} involves a {\it combination} of the $L^2$ theory for the  $\bar{\partial}$-operator (more specifically, Theorem \ref{thm:OT-extension-twisted}), pluripotential theory, and some basic geometric measure theory; see Section \ref{sect: psh extension} for details. By contrast, in much of the earlier work on extension problems for closed positive currents (particularly, complex varieties) and psh functions, the $L^2$ method and the geometric measure-theoretic method were generally used  {\it separately}, notably by Siu \cite{Siu74_Analyticity} and  by Shiffman \cite{Shiffman68}, Harvey \cite{Harvey_Extension}, and Harvey--Polking \cite{Harvey_CPAM}, respectively. The use of Theorem \ref{thm:OT-extension-twisted} in the proof of Theorem \ref{thm:Hartogs_PSH-manifold} was inspired by the work of Chen, Wu and Wang \cite{CWW} and Demailly \cite{Demailly92}. It should also be mentioned that  a weaker version of Theorem \ref{thm:Hartogs_PSH-manifold}, under additional assumptions, was established  by the author in his previous work \cite{Wang_22} by quite a different method.

We now pose the following fairly natural question concerning Theorem \ref{thm:Hartogs_PSH-manifold}:

\begin{question}\label{Conj:psh_Hartogs}
Does Theorem \ref{thm:Hartogs_PSH-manifold} still hold when $K$ is only assumed to be a compact pluripolar subset of $\Omega$?
\end{question}

The answer to this question is likely to be positive, although it remains out of reach at the time of writing. We also refer the interested reader to
\cite[Conjecture 4.1]{CWW} for a  different but closely related problem.

 Finally, we hope that the results presented in this paper (as well as in \cite{CWW}) will, to some extent, serve as an incentive
to investigate the Ohsawa problem and related topics.

The paper is organized as follows. In Section \ref{sect:preliminaries} we review some necessary notions and results that will be used in this paper. We then prove in Section \ref{sect:Ohsawa--Takegoshi existence}  an Ohsawa--Takegoshi-type $L^2$ existence theorem for the $\bar{\partial}$-operator on complete K\"ahler manifolds, which is in turn used in Section \ref{sect:Ohsawa--Takegoshi extension} to prove Theorem \ref{thm:OT-extension-twisted}. Finally, the proof of Theorem $\ref{thm:Hartogs_PSH-manifold}$ is carried out in Section \ref{sect: psh extension}.

\medskip
\noindent {\bf Acknowledgements.}
The author is very grateful to Professor Bo-Yong Chen for bringing the Ohsawa problem to his attention and for many helpful discussions. The author would also like to thank Professor Xiangyu Zhou---his mentor when he was a postdoctoral fellow at the Institute of Mathematics, AMSS, Chinese Academy of Sciences from 2017 to 2019---for his constant support and encouragement over the years. In addition, thanks go to Professor Xiankui Meng for interesting exchanges about compact complete pluripolar sets, as well as Professors Taeyong Ahn, Zhuo Liu, Daowei Ma, Takeo Ohsawa and Xiangyu Zhou for their interest in this work. Last but not least, the author is greatly indebted to the
referees for carefully reading the manuscript and suggesting many improvements.


\medskip

\section{Preliminaries}\label{sect:preliminaries}
In this section we fix some necessary notation and collect some fundamental results from pluripotential theory and the $L^2$ theory of the $\bar\partial$-operator that will be used in subsequent sections.

\subsection{Complete pluripolar sets and their defining functions}\label{sect:Complete pluripolar sets}
We begin by recalling the definition of complete pluripolar sets.  Let $X$ be a complex manifold and ${\rm Psh}(X)$ denote the set of all psh functions on $X$.

\begin{definition}
A set $E\subset X$ is said to be {\it complete pluripolar}  if for every point $x\in E$ there exists a neighborhood $U$ of $x$ and a function $\varphi\in {\rm Psh}(U)$ such that
  $$
  E\cap U=\varphi^{-1}(-\infty).
  $$
\end{definition}

The collection of all complex subvarieties of $X$ constitutes a particularly important class of (closed) complete pluripolar sets, but complete pluripolar sets are much more general: for instance, the Cartesian product of finitely many (possibly distinct) Cantor-type sets of logarithmic capacity zero in the complex plane is a {\it compact} complete pluripolar set in the corresponding complex Euclidean space (see, e.g., \cite{Ransfordbook}), but is obviously not complex-analytic. One may also consult \cite{El-Mir_cmp-plolar, LMP_IUMJ, Edlund_cmp-curv, DM_cmp-plolar} and the references therein for many other nontrivial, constructive examples of {\it compact} complete pluripolar sets in $\mathbb C^n$, especially those of minimal Hausdorff codimension (i.e., of Hausdorff codimension two). Using these examples one can easily construct a large number of compact complete pluripolar sets with positive (finite or infinite) Hausdorff $(2n-2)$-measure in general Stein manifolds of dimension $n$, where $n$ is an arbitrary positive integer.

On Stein manifolds, Col\c{t}oiu \cite{Coltoiu} proved the following important result concerning the existence of a global defining function for a closed complete pluripolar set.

\begin{theorem}[see \cite{Coltoiu}]\label{thm:def-complete-polar}
Let $X$ be a Stein manifold and $E\subset X$ be a closed complete pluripolar set. Then there is a function $\rho\in {\rm Psh}(X)\cap C^{\infty}(X\!\setminus\!E)$ such that $\rho^{-1}(-\infty)=E$ and $\sqrt{-1}\partial\bar{\partial}\rho>0$ on $X\!\setminus\!E$.
\end{theorem}

As we shall see later, this result is crucial in the proof of Theorem $\ref{thm:Hartogs_PSH-manifold}$.

\smallskip

\subsection{Twisted Bochner--Kodaira--Nakano inequality}\label{subsect: Twisted-BKN}
Let $(X, \omega)$ be an $n$-dimensional K\"{a}hler manifold and $(L, h)$ be a Hermitian holomorphic line bundle over $X$. Given integers $0\leq p, q\leq n$, we let $L^2(X,\, \Lambda^{p,\, q} T^{\ast}X \otimes L;\, \omega,h)$ denote the space of $L$-valued $(p, q)$-forms $u$ with measurable complex coefficients on $X$,  satisfying
  $$
  \|u\|_{\omega,\,h}\coloneqq\bigg(\int_X |u|^2_{\omega,\,h}\, {\rm d}V_{\omega}\bigg)^{1/2}<\infty,
  $$
where $|\,\ |^2_{\omega,\,h}$ is the pointwise norm on $\Lambda^{p,\, q} T^{\ast}X \otimes L$ induced by $\omega$ and $h$, and ${\rm d}V_\omega=\omega^n/n!$ is the volume form on $X$ induced by $\omega$. We also denote by  $\langle\,\, ,\,\rangle_{\omega,\,h}$ (resp. $\langle\!\langle\,\, ,\,\rangle\!\rangle_{\omega,\,h}$)  the pointwise (resp. global) inner product corresponding to $|\,\ |^2_{\omega,\,h}$ (resp. $\|\,\ \|^2_{\omega,\,h}$). Obviously $\langle\!\langle\,\, ,\,\rangle\!\rangle_{\omega,\,h}$ makes $L^2(X,\, \Lambda^{p,\, q} T^{\ast}X \otimes L;\, \omega,h)$ a Hilbert space,   which contains as a dense subspace the space $\mathcal{D}(X,\, \Lambda^{p,\, q}T^{\ast}X \otimes L)$ of $L$-valued $C^{\infty}$ $(p, q)$-forms with compact support in $X$. For our later purposes, we also need to consider the maximal weak extension of the $\bar\partial$-operator to $L^2(X,\, \Lambda^{p,\, q} T^{\ast}X \otimes L;\, \omega,h)$, which we still denote by $\bar\partial$. Then it is a closed, densely defined operator, and so is its Hilbert adjoint ${\bar\partial}^\ast_{\omega,\,h}$ with respect to $\langle\!\langle\,\, ,\,\rangle\!\rangle_{\omega,\,h}$.

Sometimes it is necessary to consider $L$-valued $(p, q)$-forms that are locally square integrable on $X$, rather than (globally) square integrable. The space consisting of all such forms is obviously independent of the choice of the $C^{\infty}$ metrics $\omega$ and $h$, and is simply denoted  by $L^2_{\rm loc}(X,\, \Lambda^{p,\, q} T^{\ast}X\otimes L)$.

We now record the following a priori inequality due to Ohsawa--Takegoshi \cite{OhsawaTakegoshi87} (see also \cite{Ohsawa95, Demailly-Manivel_extension, Demaillybook2}), which plays a fundamental role in the whole $L^2$ theory of the $\bar\partial$-operator.

\begin{lemma}\label{lem:Twisted BKN ineq}
Let $(X, \omega)$ be an $n$-dimensional K\"{a}hler manifold and $(L, h)$ be a Hermitian holomorphic line bundle over $X$. Given two $C^{\infty}$ functions $\eta$, $\lambda>0$ on $X$, we then have
\begin{equation}\label{ineq: Ohsawa-inequality}
\begin{split}
 \big\|&\sqrt{\eta+\lambda}\,\bar{\partial}^\ast_{\omega,\,h} u\big\|^2_{\omega,\,h}+\big\|\sqrt{\eta}\bar{\partial}u\big\|^2_{\omega,\,h}\\
  &\;\geq \big\langle\negmedspace\big\langle \big[\eta \sqrt{-1}\Theta_h(L)-\sqrt{-1}\partial\bar{\partial}\eta
         -\lambda^{-1}\sqrt{-1}\partial\eta\wedge\bar{\partial}\eta,\, \mathit{\Lambda}_{\omega}\big]u,\, u\big\rangle\negmedspace\big\rangle_{\omega,\,h}
\end{split}
\end{equation}
for all $u\in\mathcal{D}(X,\, \Lambda^{n,\, q}T^{\ast}X \otimes L)$, $q\geq 1$, where $\sqrt{-1}\Theta_h(L)$ denotes the Chern curvature form of $(L, h)$ and $\mathit{\Lambda}_{\omega}$ denotes the pointwise adjoint of the Lefschetz operator $L_{\omega}=\omega\wedge\cdot$.
\end{lemma}

Inequality \eqref{ineq: Ohsawa-inequality} also applies to the case where $(L, h)$ is a Hermitian holomorphic {\it vector} bundle over $X$.  Since we only need the case of line bundles in this paper, we restrict ourselves to this particular case.

\medskip

\section{An Ohsawa--Takegoshi-type $L^2$ existence theorem}\label{sect:Ohsawa--Takegoshi existence}
In this section we present an Ohsawa--Takegoshi-type $L^2$ existence theorem for the $\bar{\partial}$-operator on complete K\"ahler manifolds, which includes as a special case Demailly's far-reaching generalization of the  $L^2$ existence theorems due to H\"{o}rmander and Andreotti--Vesentini (cf. \cite[Th\'{e}or\`{e}me 5.1]{Demailly82} and also Remark \ref{rmk:supplemental remark}(iii) below). Such an existence result is more or less known to experts, but we were unable to locate a good reference for its proof and will include a detailed argument for the reader's convenience. (We do not claim any originality here.)

To begin with, we recall the notion of singular Hermitian metrics. Let $L$ be a holomorphic line bundle over a complex manifold $X$.
\begin{definition}
A (possibly) {\it singular Hermitian metric} $h$ on $L$ is a measurable Hermitian metric such that, with respect to any smooth local frame
$e_L$ of $L$ over an open set $U\subset X$, one has
  $$
  |e_L|_h^2=e^{-\varphi}
  $$
for some $\varphi\in L^1_{\mathrm{loc}}(U)$.
\end{definition}

It is easy to see that in terms of a given smooth Hermitian metric $h_0$, every singular Hermitian metric $h$ on $L$ can be written exactly in the form $h=h_0e^{-\varphi}$, where $\varphi\in L^1_{{\rm loc}}(X)$; and the curvature current
 $$
 \sqrt {-1}\Theta_h(L)\coloneqq-\sqrt {-1}\partial\bar{\partial}\log h
 $$
of $(L,h)$ is semi-positive exactly when the weight $\varphi=-\log(h/h_0)$ of $h$ (with respect to $h_0$) satisfies
  $$
  \sqrt {-1}\Theta_{h_0}(L)+\sqrt{-1}\partial\bar{\partial}\varphi\geq 0 \quad {\rm on}\, \, X
  $$
in the sense of currents.

Given a singular Hermitian metric $h$ on $L$, we can also define $L^2(X,\, \Lambda^{p,\, q} T^{\ast}X \otimes L;\, \omega,h)$ and $\langle\!\langle\,\, ,\,\rangle\!\rangle_{\omega,\,h}$ in the same way as in Subsection \ref{subsect: Twisted-BKN}.  We now state and prove the following

\begin{theorem}\label{thm:L^2 estimate of OT-type}
Let $X$ be an $n$-dimensional complete K\"ahler manifold with a {\rm(}not necessarily complete{\rm)} K\"ahler metric $\omega$ and $L$ be a holomorphic line bundle over $X$ with a singular Hermitian metric $h$. Suppose there are bounded $C^{\infty}$ functions $\eta$, $\lambda>0$ on $X$ such that
\begin{equation}\label{ineq: curvature condition}
  \eta\sqrt {-1}\Theta_h(L)-\sqrt{-1}\partial\bar{\partial}\eta-\lambda^{-1}\sqrt{-1}\partial\eta\wedge \bar{\partial}\eta\ge \Theta
\end{equation}
in the sense of currents, where $\Theta$ is a continuous {\rm(}strictly{\rm)} positive $(1,1)$-form on $X$. Then for every $\bar{\partial}$-closed $v\in L^2_{\rm loc}(X,\, \Lambda^{n,\, q} T^{\ast}X \otimes L)$, $q\geq 1$, satisfying
  $$
  \int_X \big\langle[\Theta, \mathit{\Lambda}_{\omega}]^{-1}v,\, v\big\rangle_{\omega,\,h}\, {\rm d}V_{\omega}<\infty,
  $$
there exists $u\in L^2(X,\, \Lambda^{n,\, q-1} T^{\ast}X \otimes L;\, \omega, h)$ such that $\bar{\partial}u=v$ and
\begin{equation}\label{ineq:L^2 estimate for solution}
  \int_X (\eta+\lambda)^{-1}|u|^2_{\omega,\,h}\, {\rm d}V_{\omega}\leq \int_X \big\langle[\Theta, \mathit{\Lambda}_{\omega}]^{-1}v,\, v\big\rangle_{\omega,\,h}\, {\rm d}V_{\omega}.
\end{equation}
\end{theorem}

 We are  mainly interested in the case $q=1$, but the general statement is needed in the proof.

\begin{proof}
First of all, a standard monotonicity argument as in \cite{Demailly82} allows us to reduce the theorem to the case when $\omega$ is complete, which we shall henceforward assume.

We next make use of Demailly's approximation theorem \cite[Th\'{e}or\`{e}me 9.1]{Demailly82}, according to which there is an increasing sequence $\{h_j\}_{j\in \mathbb N}$ of $C^{\infty}$ Hermitian metrics on $L$ such that
\begin{enumerate}[leftmargin=2.0pc, parsep=2pt]
  \item [{\rm(i)}]  $h_j\nearrow h$ on $X$ as $j\to\infty$, and
  \item [{\rm(ii)}] the curvature of each $h_j$ satisfies
       \begin{equation}\label{ineq: curvature-ineq}
       \eta\sqrt{-1}\Theta_{h_j}(L)-\sqrt{-1}\partial\bar{\partial}\eta-\lambda^{-1}\sqrt{-1}\partial\eta\wedge \bar{\partial}\eta\geq \Theta-\delta_j\omega,
       \end{equation}
       where $\{\delta_j\}_{j\in \mathbb N}$ is a decreasing sequence of continuous functions on $X$ that converges to $0$ {\it almost everywhere} on $X$ as $j\to\infty$.
\end{enumerate}
Since $(X,\, \omega)$ is complete, it is  easy to construct a $C^{\infty}$ exhaustion function $\rho$ for $X$ with $|{\rm d}\rho|_{\omega}\leq 1$ on $X$. Choose also a $C^{\infty}$ function $\chi\!: \mathbb R \to [0,\, 1]$ such that  $\chi|_{(-\infty,\,1)}=1$, $\chi|_{(2,\,\infty)}=0$   and  $|\chi'|\leq 2$ on $\mathbb R$, and define
  $$
  \chi_j\coloneqq\chi(\rho/j),\quad j\in \mathbb N.
  $$

We now proceed to establish several  necessary estimates. For notational simplicity, we set
  $$
  L^2(X,\, \Lambda^{n,\, q} T^{\ast}X \otimes L)
  =L^2(X,\, \Lambda^{n,\, q} T^{\ast}X \otimes L;\, \omega, h),
  $$
  $$
  L^2(X,\, \Lambda^{n,\, q} T^{\ast}X \otimes L)_k
  =L^2(X,\, \Lambda^{n,\, q} T^{\ast}X \otimes L;\, \omega, h_k)
  $$
and
  $$
  \langle\!\langle\,\, ,\,\rangle\!\rangle=\langle\!\langle\,\, ,\,\rangle\!\rangle_{\omega,\,h},
  \quad \langle\!\langle\,\, ,\,\rangle\!\rangle_k=\langle\!\langle\,\, ,\,\rangle\!\rangle_{\omega,\,h_k}
  $$
as well as ${\bar\partial}^\ast_k={\bar\partial}^\ast_{\omega,\,h_k}$ for $k\in\mathbb N$. Similar simplifications apply to other notation, and we will omit subscripts when no confusion arises. Given $v$ as in the theorem, since the Hermitian operator $[\Theta,\mathit{\Lambda}]$ on $\Lambda^{n,\, q} T^{\ast}X \otimes L$ is positive definite (as follows from the positivity of $\Theta$; see, e.g., \cite[Chapter VI, Proposition 5.8]{Demaillybook}),  we have
\begin{align}\label{ineq:basic-estimate1}
\begin{split}
  \big|\langle\!\langle \chi_j v, w\rangle\!\rangle_k\big|^2
  &\leq \big\langle\negmedspace\big\langle [\Theta,\mathit{\Lambda}]^{-1}v,\,
        v\big\rangle\negmedspace\big\rangle_k\,
        \big\langle\negmedspace\big\langle [\Theta,\mathit{\Lambda}](\chi_j w),\,
        \chi_j w\big\rangle\negmedspace\big\rangle_k\\
  &\leq \big\langle\negmedspace\big\langle [\Theta,\mathit{\Lambda}]^{-1}v,\,
        v\big\rangle\negmedspace\big\rangle
        \big\langle\negmedspace\big\langle [\Theta,\mathit{\Lambda}](\chi_j w),\,
        \chi_j w\big\rangle\negmedspace\big\rangle_k\\
\end{split}
\end{align}
for all $w\in {\rm Dom}\,\bar{\partial}\cap {\rm Dom}\,{\bar\partial}^\ast_k$ and $j,\, k\in \mathbb N$, in view of the Cauchy--Schwarz inequality and the fact that $h_k\leq h$. Since $\chi_j w\in {\rm Dom}\,\bar{\partial}\cap
{\rm Dom}\,{\bar\partial}^\ast_k$ as is easily checked by definition, Ohsawa--Takegoshi's a priori inequality \eqref{ineq: Ohsawa-inequality} together with
\eqref{ineq: curvature-ineq} and Andreotti--Vesentini's density theorem (see, e.g., \cite[Chapter VIII]{Demaillybook}) implies
\begin{align}\label{ineq: apriori-inequality with error term}
 \begin{split}
  \big\|\sqrt{\eta+\lambda}\,\bar{\partial}^\ast_k
  (\chi_j w)\big\|^2_k+\big\|\sqrt{\eta}\,\bar{\partial}(\chi_j w)\big\|^2_k
  & \geq \big\langle\negmedspace\big\langle [\Theta-\delta_k\omega,\, \mathit{\Lambda}](\chi_j w),\, \chi_j w\big\rangle\negmedspace\big\rangle_k\\
  &=\big\langle\negmedspace\big\langle [\Theta,\, \mathit{\Lambda}](\chi_j w),\, \chi_j w\big\rangle\negmedspace\big\rangle_k-q\big\|\chi_j\delta_k^{1/2} w\big\|^2_k,
\end{split}
\end{align}
where the last equality follows from the fact that $[\omega, \mathit{\Lambda}]=q\,{\rm Id}$ on $\Lambda^{n,\, q} T^{\ast}X$. Observe also that
   $$
   \bar{\partial}(\chi_j w)=\chi_j\bar{\partial}w+\bar{\partial}\chi_j \wedge w
   $$
and
   $$
   \bar{\partial}^\ast_k(\chi_j w)=\chi_j\bar{\partial}^\ast_kw
   -({\rm grad}\, \chi_j)^{0,\, 1}\lrcorner\, w,
   $$
where $\lrcorner$ denotes the contraction operator and $({\rm grad}\, \chi_j)^{0,\, 1}$ denotes the $(0,1)$-part of the gradient of $\chi_j$ under the metric corresponding to $\omega$. Substituting these two identities into \eqref{ineq: apriori-inequality with error term} and applying the Cauchy--Schwarz inequality, we get
\begin{align}\label{ineq:basic-estimate2}
\begin{split}
  \big\langle\negmedspace\big\langle [\Theta,\, \mathit{\Lambda}](\chi_j w),\, \chi_j w
  \big\rangle\negmedspace\big\rangle_k
  &\leq q\big\|\chi_j\delta_k^{1/2} w\big\|^2_k
   +\big(1+\frac 1j\big) \Big(\big\|\sqrt{\eta+\lambda}\,\bar{\partial}^\ast_k
         w\big\|^2_k+\big\|\sqrt{\eta}\,\bar{\partial}w\big\|^2_k\Big)\\
  &\quad\; +(1+j)\Big(\big\|\sqrt{\eta+\lambda}\,({\rm grad}\, \chi_j)^{0,\, 1}\lrcorner\,
   w\big\|^2_k+\big\|\sqrt{\eta}\,\bar{\partial}\chi_j \wedge w\big\|^2_k\Big)\\
  &\leq \big(1+\frac1j\big)\Big(
     \big\|\sqrt{\eta+\lambda}\,\bar{\partial}^\ast_k w\big\|^2_k
     +\big\|\sqrt{\eta}\,\bar{\partial}w\big\|^2_k
     +\frac Cj \|w\|^2_k\Big)\\
  &\quad\; +q\big\|\chi_j\delta_k^{1/2} w\big\|^2_k
\end{split}
\end{align}
with $C=2 \sup_X(2\eta+\lambda)$, where in the last inequality we have used the facts that
  $$
  \big|\bar{\partial}\chi_j \wedge w\big|_k
  \leq |\bar{\partial}\chi_j||w|_k
  \leq \frac{\sqrt 2}{j} |w|_k
  $$
and
  $$
  \big|({\rm grad}\, \chi_j)^{0,\, 1}\lrcorner\, w\big|_k
  \leq |\bar{\partial}\chi_j||w|_k
  \leq \frac{\sqrt 2}{j} |w|_k.
  $$
Now combining \eqref{ineq:basic-estimate1} with \eqref{ineq:basic-estimate2} we conclude that for $j,\, k\in \mathbb N$ the estimate
\begin{equation}\label{ineq:basic-estimate3}
  \big|\langle\!\langle \chi_j v, w\rangle\!\rangle_k\big|^2
    \leq C_j\Big(\big\|\sqrt{\eta+\lambda}\,\bar{\partial}^\ast_k w\big\|^2_k
                   +\big\|\sqrt{\eta}\,\bar{\partial}w\big\|^2_k
             +q\big\|\chi_j\delta_k^{1/2} w\big\|^2_k+\frac{C}{j}\|w\|^2_k\Big)
\end{equation}
holds for all $w\in {\rm Dom}\,\bar{\partial}\cap {\rm Dom}\,{\bar\partial}^\ast_k$, with
   $$
   C_j\coloneqq\big(1+\frac 1j\big)\big\langle\negmedspace\big\langle [\Theta,\mathit{\Lambda}]^{-1}v,\,
        v\big\rangle\negmedspace\big\rangle.
   $$
This estimate enables us to prove the theorem by descending induction on $q$, as we shall now explain. The idea is due to Demailly \cite{Demailly82}, as mentioned in the Introduction.

\smallskip
When $q=n$, \eqref{ineq:basic-estimate3} reads
   $$
   \big|\langle\!\langle \chi_j v, w\rangle\!\rangle_k\big|^2
    \leq C_j\Big( \big\|\sqrt{\eta+\lambda}\,\bar{\partial}^\ast_k w\big\|^2_k
                 +q\big\|\chi_j\delta_k^{1/2} w\big\|^2_k+\frac{C}{j}\|w\|^2_k \Big)
   $$
for all $w\in {\rm Dom}\,{\bar\partial}^\ast_k$. This shows that the anti-linear functional
   $
   w\mapsto \langle\!\langle \chi_j v, w\rangle\!\rangle_k
   $
is continuous on the Hilbert space
   $$
   \Big({\rm Dom}\,{\bar\partial}^\ast_k,\,
   \big\|\sqrt{\eta+\lambda}\,\bar{\partial}^\ast_k\cdot\big\|^2_k
                   +q\big\|\chi_j\delta_k^{1/2} \cdot\big\|^2_k+\frac{C}{j}\|\cdot\|^2_k
   \Big)
   $$
with norm at most $\sqrt{C_j}$. Hence by the Riesz representation theorem, there exist
   $$
   u_{jk}\in L^2(X,\, \Lambda^{n,\, n-1} T^{\ast}X \otimes L)_k \quad {\rm and}
   \quad   v_{jk},\, w_{jk}\in L^2(X,\, \Lambda^{n,\, n} T^{\ast}X \otimes L)_k
   $$
such that
\begin{equation}\label{ineq:norm-estimate1}
   \big\|(\eta+\lambda)^{-\frac12}u_{jk}\big\|^2_k+\|v_{jk}\|^2_k
   +\|w_{jk}\|^2_k\leq C_j
\end{equation}
and
   $$
   \langle\!\langle u_{jk}, {\bar\partial}^\ast_k w \rangle\!\rangle_k
    +\sqrt{q}\,\big\langle\negmedspace\big\langle v_{jk},\,\chi_j\delta_k^{1/2} w\big\rangle\negmedspace\big\rangle_k
    +\sqrt{C/j}\,\langle\!\langle w_{jk}, w\rangle\!\rangle_k
    =\langle\!\langle \chi_j v, w\rangle\!\rangle_k
   $$
for all $w\in {\rm Dom}\,{\bar\partial}^\ast_k$. This means that
$u_{jk}\in {\rm Dom}\,{\bar\partial}^{\ast\ast}_k={\rm Dom}\, \bar\partial$ and
   $$
   \bar\partial u_{jk}+\sqrt{q}\,\chi_j\delta_k^{1/2} v_{jk}+\sqrt{C/j}\,w_{jk}=\chi_j v
   $$
for all $j,\, k\in \mathbb N$. Moreover, since $\{h_k\}_{k\in \mathbb N}$ is increasing,  \eqref{ineq:norm-estimate1} and the Banach--Alaoglu theorem imply that $\{u_{jk}\}_{k\in\mathbb N}$ (resp. $\{w_{jk}\}_{k\in\mathbb N}$) has a subsequence that converges weakly to some $u_j$ (resp. $w_j$) in $L^2_{\rm loc}(X,\,\Lambda^{n,\, n-1} T^{\ast}X \otimes L)$ (resp. $L^2_{\rm loc}(X,\,\Lambda^{n,\, n} T^{\ast}X \otimes L)$). Also since $\{\delta_k\}_{k\in \mathbb N}$  decreases to $0$ almost everywhere on $X$, it follows readily from \eqref{ineq:norm-estimate1}, the Cauchy--Schwarz inequality and the dominated convergence theorem that $\big\{ \chi_j\delta_k^{1/2} v_{jk}\big\}_{k\in\mathbb N}$ converges to $0$ in $L^1_{\rm loc}(X,\,\Lambda^{n,\, n} T^{\ast}X \otimes L)$. Consequently
   $$
   \bar\partial u_j+\sqrt{C/j}\,w_j=\chi_j v
   $$
and
 \begin{align*}
  \begin{split}
    \int_{\Omega} \big((\eta+\lambda)^{-1}|u_j|^2_l+|w_j|^2_l\big)\, {\rm d}V
    &\leq \limsup\limits_{k\to\infty} \int_{\Omega}
          \big((\eta+\lambda)^{-1}|u_{jk}|^2_l+|w_{jk}|^2_l\big)\, {\rm d}V\\
    &\leq \limsup\limits_{k\to\infty}
           \big(\big\|(\eta+\lambda)^{-\frac12}u_{jk}\big\|^2_k
           +\|w_{jk}\|^2_k\big)\\
    &\leq C_j
  \end{split}
\end{align*}
for all $j,\, l\in \mathbb N$ and open sets $\Omega\subset\subset X$. Letting first $l\to \infty$ and then $\Omega\to X$ we obtain
  $$
  \big\|(\eta+\lambda)^{-\frac12}u_j\big\|^2+\|w_j\|^2 \leq C_j
  $$
for all $j\in \mathbb N$, in view of the monotone convergence theorem. Taking a weak limit point $u$ of $\{u_j\}_{j\in\mathbb N}$
in $L^2_{\rm loc}(X,\,\Lambda^{n,\, n-1} T^{\ast}X \otimes L)$, we then conclude  that $\bar{\partial}u=v$ and
  $$
  \big\|(\eta+\lambda)^{-\frac12}u\big\|^2
  \leq \lim\limits_{j\to\infty} C_j
  =\big\langle\negmedspace\big\langle [\Theta,\mathit{\Lambda}]^{-1}v,\,
        v\big\rangle\negmedspace\big\rangle.
  $$
This completes the proof of the theorem in the case $q=n$.

\smallskip
Now assume that $1\leq q\leq n-1$ and that the conclusion of the theorem is true for all $\bar{\partial}$-closed $(n, q+1)$-forms with the prescribed integrability. We proceed to show that the conclusion is also true for all $(n, q)$-forms with the same property. Given such a form $v$ and $j\in \mathbb N$, consider the $\bar{\partial}$-closed $(n, q+1)$-form $\bar\partial (\chi_j v)=\bar{\partial}\chi_j \wedge v$ which satisfies
  \begin{equation*}
       \big\langle\negmedspace\big\langle  [\Theta,\mathit{\Lambda}]^{-1} \bar\partial (\chi_j v), \, \bar\partial (\chi_j v)  \big\rangle\negmedspace\big\rangle
  \leq \sup_X |\bar{\partial}\chi_j|^2\, \big\langle\negmedspace\big\langle  [\Theta,\mathit{\Lambda}]^{-1}v,\, v \big\rangle\negmedspace\big\rangle
  \leq \frac{2}{j^2}\big\langle\negmedspace\big\langle   [\Theta,\mathit{\Lambda}]^{-1}v,\, v  \big\rangle\negmedspace\big\rangle<\infty,
  \end{equation*}
where the first inequality follows from \cite[Lemme 3.2 (3.5)]{Demailly82}. By this estimate and the induction hypothesis, there exists $v_j\in L^2(X,\, \Lambda^{n,\, q} T^{\ast}X \otimes L)$ such that $\bar\partial v_j=\bar\partial (\chi_j v)$ and
   \begin{equation}\label{ineq:modification-estimate}
    \big\|(\eta+\lambda)^{-\frac12} v_j\big\|^2 \leq \frac{2}{j^2} \big\langle\negmedspace\big\langle
        [\Theta,\mathit{\Lambda}]^{-1}v,\, v\big\rangle\negmedspace\big\rangle.
   \end{equation}
In particular, $\chi_j v-v_j\in {\rm Ker}\,\bar\partial$ and
  $$
  v_j\in L^2(X,\, \Lambda^{n,\, q} T^{\ast}X \otimes L)
  \subset L^2(X,\, \Lambda^{n,\, q} T^{\ast}X \otimes L)_k.
  $$
On the other hand, note that every $w\in {\rm Dom}\,{\bar\partial}^\ast_k$ admits an orthogonal decomposition $w=w_1+w_2$, where
  $$
  w_1\in {\rm Ker}\,\bar\partial \quad {\rm and} \quad
  w_2\in ({\rm Ker}\,\bar\partial)^\perp \subset {\rm Ker}\,{\bar\partial}^\ast_k.
  $$
It follows that
  $$
  \langle\!\langle \chi_j v-v_j, w\rangle\!\rangle_k
  =\langle\!\langle \chi_j v-v_j, w_1\rangle\!\rangle_k,
  $$
and hence
  $$
  \langle\!\langle \chi_j v, w\rangle\!\rangle_k
  =\langle\!\langle \chi_j v, w_1\rangle\!\rangle_k
    +\langle\!\langle v_j, w_2\rangle\!\rangle_k.
  $$
Applying the Cauchy--Schwarz inequality to this and using \eqref{ineq:modification-estimate}, we get
\begin{align*}\label{ineq:basic-estimate4}
  \begin{split}
  \big|\langle\!\langle \chi_j v, w\rangle\!\rangle_k\big|^2
  &\leq \big(1+\frac1j\big)
    \big|\langle\!\langle \chi_j v, w_1\rangle\!\rangle_k\big|^2
     +(1+j)\|v_j\|^2 \|w_2\|^2_k\\
  &\leq \big(1+\frac1j\big)
    \Big(\big|\langle\!\langle \chi_j v, w_1\rangle\!\rangle_k\big|^2
         +\frac2j \sup_X(\eta+\lambda) \big\langle\negmedspace\big\langle
         [\Theta,\mathit{\Lambda}]^{-1}v,\, v\big\rangle\negmedspace\big\rangle
         \|w\|^2_k\Big).
  \end{split}
\end{align*}
In addition, since $w_1\in {\rm Ker}\,\bar{\partial}\cap {\rm Dom}\,{\bar\partial}^\ast_k$ and $w_2\in {\rm Ker}\,{\bar\partial}^\ast_k$, applying \eqref{ineq:basic-estimate3} to $w_1$ yields
\begin{align*}
  \begin{split}
   \big|\langle\!\langle \chi_j v, w_1\rangle\!\rangle_k\big|^2
   &\leq C_j\Big(\big\|\sqrt{\eta+\lambda}\,\bar{\partial}^\ast_k w\big\|^2_k
                   +q\big\|\chi_j \delta_k^{1/2} w_1\big\|^2_k+\frac{C}{j}\|w\|^2_k\Big).
  \end{split}
\end{align*}
Substituting this estimate into the preceding one we arrive at
\begin{equation*}\label{ineq: Final basic-estimate}
  \big|\langle\!\langle \chi_j v, w\rangle\!\rangle_k\big|^2
  \leq \big(1+\frac1j\big)C_j\Big(
          \big\|\sqrt{\eta+\lambda}\,\bar{\partial}^\ast_k w\big\|^2_k
         + q \big\|\chi_j \delta_k^{1/2} P_k w \big\|^2_k+ \frac{2C}{j}\|w\|^2_k
        \Big)
\end{equation*}
for all $w\in {\rm Dom}\,{\bar\partial}^\ast_k$, where $P_k$ denotes the orthogonal projection from $L^2(X,\, \Lambda^{n,\, q} T^{\ast}X \otimes L)_k$ onto ${\rm Ker}\,\bar\partial$. Then the Riesz representation theorem implies that there exist
   $$
   u_{jk}\in L^2(X,\, \Lambda^{n,\, q-1} T^{\ast}X \otimes L)_k
   \quad {\rm and} \quad
   v_{jk},\, w_{jk}\in L^2(X,\, \Lambda^{n,\, q} T^{\ast}X \otimes L)_k
   $$
such that
\begin{equation}\label{ineq:norm-estimate2}
   \big\|(\eta+\lambda)^{-\frac12}u_{jk}\big\|^2_k+\|v_{jk}\|^2_k
   +\|w_{jk}\|^2_k\leq \big(1+\frac1j\big) C_j
\end{equation}
and
   $$
   \langle\!\langle u_{jk}, {\bar\partial}^\ast_k w \rangle\!\rangle_k
    +\sqrt{q}\,\big\langle\negmedspace\big\langle v_{jk},\,\chi_j \delta_k^{1/2}
     P_k w\big\rangle\negmedspace\big\rangle_k
    +\sqrt{2C/j}\,\langle\!\langle w_{jk}, w\rangle\!\rangle_k
    =\langle\!\langle \chi_j v, w\rangle\!\rangle_k
   $$
for all $w\in {\rm Dom}\,{\bar\partial}^\ast_k$. It follows that
   $$
   \bar\partial u_{jk}+\sqrt{q} P_k(\chi_j \delta_k^{1/2} v_{jk})+\sqrt{2C/j}\,w_{jk}=\chi_j v
   $$
for all $j,\, k\in \mathbb N$. To complete the proof, it remains to show that for every $j\in \mathbb N$,
\begin{equation}\label{eq:weak-convergence}
  f_{jk}\coloneqq P_k(\chi_j\delta_k^{1/2} v_{jk})\to 0 \quad {\rm weakly\, \, in}\quad
  L^2_{\rm loc}(X,\,\Lambda^{n,\, q} T^{\ast}X \otimes L)
\end{equation}
as $k\to \infty$. Once this is established, the same limiting argument as in the latter half of the preceding paragraph yields  the existence of a solution $u$ to equation $\bar{\partial}u=v$ together with the desired estimate
  $$
  \big\|(\eta+\lambda)^{-\frac12}u\big\|^2
  \leq \big\langle\negmedspace\big\langle [\Theta,\mathit{\Lambda}]^{-1}v,\, v\big\rangle\negmedspace\big\rangle,
  $$
thus completing the induction procedure (and hence the proof of the theorem).

\smallskip
To prove \eqref{eq:weak-convergence}, we argue as follows. First observe that by \eqref{ineq:norm-estimate2},
\begin{equation}\label{ineq:norm-estimate3}
  \|f_{jk}\|_k\leq \big\|\chi_j\delta_k^{1/2} v_{jk}\big\|_k
  \leq \sup_X (\chi_j\delta_1^{1/2})\,\|v_{jk}\|_k\leq {\rm const}_j
\end{equation}
for all $k\in\mathbb N$, where ${\rm const}_j$ denotes a constant depending only on $j$. This observation is crucial in the following argument and will be used repeatedly. For the moment, this bound, together with  the Banach--Alaoglu theorem, implies  that the set of weak limit points of $\{f_{jk}\}_{k\in\mathbb N}$ is non-empty. Thus it suffices to show that $0$ is the only element of this set. To establish the validity of the latter, we may assume without loss of generality that the sequence $\{f_{jk}\}_{k\in\mathbb N}$ itself is weakly convergent in  $L^2_{\rm loc}(X,\,\Lambda^{n,\, q} T^{\ast}X \otimes L)$, with limit $f_j$, say. Let $H_k$ (resp. $H$) denote the only positive definite Hermitian operator on $(L, h_1)$ that satisfies
 $|H_k\zeta|_1=|\zeta|_k$ (resp. $|H\zeta|_1=|\zeta|$) for all $\zeta\in L$.
Then
  $$
  {\rm Id}_L\leq H_k\leq H_{k+1}
  $$
for all $k\in\mathbb N$, and $\{H_k\}_{k\in\mathbb N}$ converges to $H$ almost everywhere on $X$. After extending $H_k$ and $H$ naturally to $\Lambda^{n,\, q} T^{\ast}X \otimes L$,  we claim that
  $$
  H_k f_{jk}\to Hf_j \quad {\rm weakly\, \, in}\quad
  L^2_{\rm loc}(X,\,\Lambda^{n,\, q} T^{\ast}X \otimes L)
  $$
as $k\to \infty$. To show this, it suffices to check that $Hf_j$ is the only weak limit point of
$\{H_k f_{jk}\}_{k\in\mathbb N}$. Again by \eqref{ineq:norm-estimate3} and the Banach--Alaoglu theorem, we may assume that  the sequence $\{H_k f_{jk}\}_{k\in\mathbb N}$ itself is weakly convergent in  $L^2_{\rm loc}(X,\,\Lambda^{n,\, q} T^{\ast}X \otimes L)$, with limit $g_j$, say.
Observe that for every compactly supported
$\varphi\in L^2(X,\,\Lambda^{n,\, q} T^{\ast}X \otimes L)_1$,
   $$
   \big|\big\langle\negmedspace\big\langle H_k f_{jk},\, (H_k^{-1}-H^{-1})\varphi\big\rangle\negmedspace\big\rangle_1\big|
   \leq \|f_{jk}\|_k\big\|(H_k^{-1}-H^{-1})\varphi\big\|_1
   \leq {\rm const}_j \big\|(H_k^{-1}-H^{-1})\varphi\big\|_1
   \to 0
   $$
as $k\to \infty$, in view of \eqref{ineq:norm-estimate3}, the Cauchy--Schwarz inequality and the dominated convergence theorem. Thus
  $$
  \langle\!\langle f_j, \varphi\rangle\!\rangle_1
   =\lim\limits_{k\to \infty} \langle\!\langle H_k f_{jk},\, H_k^{-1} \varphi\rangle\!\rangle_1
   =\lim\limits_{k\to \infty} \langle\!\langle H_k f_{jk},\, H^{-1} \varphi\rangle\!\rangle_1
   =\langle\!\langle g_j, H^{-1} \varphi \rangle\!\rangle_1,
  $$
hence $g_j=Hf_j$ and the claim follows. Moreover by \eqref{ineq:norm-estimate3} and the definition of $H_k$, we see that $Hf_j\in L^2(X,\,\Lambda^{n,\, q} T^{\ast}X \otimes L)_1$. Now using \eqref{ineq:norm-estimate3} again and a standard $\varepsilon/3$-argument, we further conclude that
  $$
  H_k f_{jk}\to Hf_j \quad {\rm weakly\, \, in}\quad
  L^2 (X,\,\Lambda^{n,\, q} T^{\ast}X \otimes L)_1
  $$
as $k\to \infty$. In particular
   $$
   \|f_j\|^2=\|Hf_j\|^2_1=\lim\limits_{k\to \infty} \langle\!\langle H_k f_{jk}, Hf_j\rangle\!\rangle_1
   =\lim\limits_{k\to \infty} \langle\!\langle H_k f_{jk}, H_kf_j\rangle\!\rangle_1,
   $$
where the last equality follows from the fact that
   $$
   \big|\big\langle\negmedspace\big\langle H_k f_{jk},\, (H_k-H)f_j\big\rangle\negmedspace\big\rangle_1\big|
   \leq \|f_{jk}\|_k\big\|(H_k-H)f_j\big\|_1
   \leq {\rm const}_j \big\|(H_k-H)f_j\big\|_1
   \to 0
   $$
as $k\to \infty$. Equivalently,
   $$
   \|f_j\|^2=\lim\limits_{k\to \infty} \langle\!\langle f_{jk}, f_j\rangle\!\rangle_k.
   $$
On the other hand, since $\bar{\partial} f_{jk}=0$ for all $k\in\mathbb N$, it follows that
$\bar{\partial} f_j=0$ as well, hence
   $$
   \langle\!\langle f_{jk}, f_j\rangle\!\rangle_k
   =\big\langle\negmedspace\big\langle \chi_j\delta_k^{1/2} v_{jk},\,
     f_j\big\rangle\negmedspace\big\rangle_k.
   $$
Moreover
   $$
   \big|\langle\!\langle \chi_j\delta_k^{1/2} v_{jk},\, f_j\rangle\!\rangle_k\big|
   \leq \|v_{jk}\|_k \big\|\chi_j f_j\delta_k^{1/2}\big\|_k
   \leq {\rm const}_j \big\|\chi_j f_j\delta_k^{1/2}\big\|
   \to 0
   $$
as $k\to \infty$, in view of \eqref{ineq:norm-estimate2} and the dominated convergence theorem. Consequently $\|f_j\|=0$, i.e., $f_j=0$ almost everywhere on $X$. This establishes the desired weak convergence of $\{f_{jk}\}_{k\in\mathbb N}$ and hence the theorem.
\end{proof}

\begin{remark}\label{rmk:supplemental remark}
A few remarks are in order:
\begin{enumerate}[leftmargin=2.0pc, parsep=2pt]
  \item [{\rm(i)}] When $q=1$, the estimate in \eqref{ineq:L^2 estimate for solution} is independent of the choice of the K\"{a}hler metric $\omega$. Indeed, Demailly \cite{Demailly82} and Ohsawa  \cite{Ohsawa95} have respectively explained this in detail for the right-hand side of \eqref{ineq:L^2 estimate for solution}; for the left-hand side, the reason is simpler, as one readily sees that $|\,\ |^2_{\omega,\,h}\,{\rm d}V_{\omega}$ can be expressed intrinsically in terms of $\{\,\, ,\,\}_h$ (to be recalled in Section \ref{sect:Ohsawa--Takegoshi extension}) as follows:
      \begin{equation}\label{eq:intr-formula}
      |u|^2_{\omega,\,h}\,{\rm d}V_{\omega}=\sqrt{-1}^{\,n^2}\{u, u\}_h.
      \end{equation}
  Moreover, when $q=1$, if the form $v$ in Theorem \ref{thm:L^2 estimate of OT-type} is smooth then every solution $u$ of the equation $\bar{\partial}u=v$ is smooth as well. This follows immediately from elliptic regularity for the $\bar{\partial}$-operator on $(n,0)$-forms.

  \item [{\rm(ii)}]  In \cite{Demailly-Manivel_extension} (see also \cite{Ohsawa95, Demaillybook2}), Demailly proved using Ohsawa--Takegoshi's a priori inequality \eqref{ineq: Ohsawa-inequality} an abstract $L^2$ existence theorem for the $\bar{\partial}$-operator under the assumption that the metric $h$ on $L$ is {\it smooth} and $v$ satisfies the additional $L^2$ integrability
      $$
      \int_X |v|^2_{\omega,\,h}\, {\rm d}V_{\omega}<\infty.
      $$
      Compared with this result, the point of Theorem \ref{thm:L^2 estimate of OT-type} is that it allows the metric $h$ to be \textit{singular}. This is crucial in proving  Theorem $\ref{thm:OT-extension-twisted}$ (and hence Theorem \ref{thm:Hartogs_PSH-manifold}).

  \item [{\rm(iii)}] As the preceding proof indicates, the curvature condition \eqref{ineq: curvature condition} can be weakened to a condition analogous to that in \cite[Th\'{e}or\`{e}me 5.1]{Demailly82} and $L$ may also be replaced by a higher-rank vector bundle. To avoid technicality, we restrict ourselves to the present case. As we shall see in the next section, Theorem \ref{thm:L^2 estimate of OT-type} is sufficient for the purpose of the present paper.

  \item [{\rm(iv)}] In the special case when $X$ is a bounded complete K\"ahler domain in $\mathbb C^n$ and
      $$
      (L, h)=(X\times \mathbb C,\, e^{-\varphi})
      $$
      with $\varphi\in {\rm Psh}(X)$, Theorem \ref{thm:L^2 estimate of OT-type} was first proved by Chen--Wu--Wang\footnote{They only considered the case when $v$ is smooth and of type $(n,1)$, but their method actually works equally well for not necessarily smooth $(n,q)$-forms $v$, at least under the additional integrability assumption $\int_X |v|^2_{\omega}\, e^{-\varphi} {\rm d}V_{\omega}<\infty$.}
      in \cite{CWW} using a different method, namely by first solving the twisted Laplace equation with Dirichlet boundary condition on a sequence of smooth subdomains increasing to $X$ and then using these solutions to approximate the desired solution of the twisted  $\bar{\partial}$-equation on $X$ with $L^2$ estimate. Although quite interesting, this approach does not seem to work in the manifold case, since -- as far as the author can see -- it requires that the {\it strict positivity} of the curvature term on the left-hand side of \eqref{ineq: curvature condition} be preserved during the regularization process of the singular Hermitian metric $h$; but in general the loss of positivity (i.e., the error term $\delta_j\omega$ in \eqref{ineq: curvature-ineq}) arising from any regularization sequence $\{h_j\}_{j\in \mathbb N}$ is inevitable and difficult to control even locally.
\end{enumerate}
\end{remark}

\medskip

\section{Ohsawa--Takegoshi $L^2$ extension theorems: the complete K\"{a}hler case}\label{sect:Ohsawa--Takegoshi extension}

Let $\Omega$ be a domain in a complex manifold and $L$ be a holomorphic line bundle over $\Omega$. Denote by $C^{\infty}(\Omega,\, \Lambda^{p,\, q} T^{\ast}\Omega \otimes L)$ the set of all $L$-valued $C^{\infty}$ $(p, q)$-forms on $\Omega$. Given a $C^{\infty}$ Hermitian metric $h$ on $L$, there is a canonical sesquilinear map
  $$
  \{\,\, ,\,\}_h\!: C^{\infty}(\Omega,\, \Lambda^{p,\, q} T^{\ast}\Omega \otimes L)
  \times C^{\infty}(\Omega,\, \Lambda^{r,\, s} T^{\ast}\Omega \otimes L)
  \to C^{\infty}(\Omega,\, \Lambda^{p+s,\, q+r} T^{\ast}\Omega)
  $$
that combines the wedge product of scalar-valued forms with the metric $h$ on $L$. More explicitly, for any given $u\in C^{\infty}(\Omega,\, \Lambda^{p,\, q} T^{\ast}\Omega \otimes L)$ and
$v\in C^{\infty}(\Omega,\, \Lambda^{r,\, s} T^{\ast}\Omega \otimes L)$, if we write
  $$
  \left.u\right|_U=f\otimes\zeta \quad {\rm and} \quad \left.v\right|_U=g\otimes\zeta
  $$
in terms of a frame $\zeta$ for $L$ over an open subset $U$ of $\Omega$, then
  $$
  \left.\{u, v\}_h\right|_U=f\wedge \bar{g}\, |\zeta|^2_h.
  $$
Clearly the definition of $\{\,\, ,\,\}_h$ is pointwise and hence makes sense for $u\in\Lambda^{p,\, q} T^{\ast}_{x_0}\Omega \otimes L_{x_0}$ and $v\in \Lambda^{r,\, s} T^{\ast}_{x_0}\Omega \otimes L_{x_0}$, $x_0\in \Omega$, even when $h$ is merely a singular Hermitian metric on $L$.

With the above necessary notion we can now turn to the main goal of this section, namely proving Theorem \ref{thm:OT-extension-twisted}, whose statement we recall for ease of reference.

\begin{theorem}\label{thm:OT-extension-twisted-restated}
Let $X$ be an $n$-dimensional Stein manifold with a continuous volume form ${\rm d}V_X$ and $\Omega$ be a relatively compact complete K\"ahler domain in $X$. Suppose $L$ is a holomorphic line bundle over $\Omega$ and $h$ is a singular Hermitian metric on $L$ with semi-positive curvature current. Then for every $x_0\in \Omega$ and $s_{x_0}\in \Lambda^n T^{\ast}_{x_0}\Omega \otimes L_{x_0}$, there is an $L$-valued holomorphic $n$-form $s$ on $\Omega$ such that $s(x_0)=s_{x_0}$ and
   $$
   \int_{\Omega} \sqrt{-1}^{\,n^2}\{s, s\}_h \leq {\rm const}_{n,\, \Omega,\, {\rm d}V_X}\frac{\sqrt{-1}^{\,n^2}\{s_{x_0}, s_{x_0}\}_h}{{\rm d}V_X(x_0)}
   $$
provided the right-hand side is finite.
\end{theorem}

Here and henceforward we use ${\rm const}_{n,\, \Omega,\, {\rm d}V_X}$ to denote a generic constant that depends only on $n$, $\Omega$ and ${\rm d}V_X$, and may be different in different appearances. We shall also use similar notation (e.g. ${\rm const}_n$) with obvious meaning in the context.

\smallskip
We now prove the above theorem.  Since $\overline{\Omega}$ is compact, we can assign to every point $x\in\Omega$ a holomorphic local coordinate chart $z_x=(z_x^1,\ldots,z_x^n)$ of $X$ defined on a neighborhood of the closure of $B_x\coloneqq\{|z_x|<1\}$ such that $z_x(x)=0$ and
\begin{equation}\label{ineq:vol-comp}
  \frac{\sqrt{-1}^{\,n^2} {\rm d}z_x\wedge {\rm d}\bar z_x}{{\rm d}V_X}\geq {\rm const}_{n,\, \Omega,\, {\rm d}V_X}\quad {\rm on}\, \, B_x,
\end{equation}
where ${\rm d}z_x\coloneqq {\rm d}z_x^1\wedge\cdots\wedge {\rm d}z_x^n$. Set
  $$
  \varepsilon B_x=\{|z_x|<\varepsilon \}
  $$
for $x\in \Omega$ and $\varepsilon\in (0,1)$. We claim the following

\begin{lemma}\label{lem:scaling OT}
Given $x_0\in \Omega$, $f_{x_0}\in \Lambda^n T^{\ast}_{x_0}\Omega$ and $\varepsilon\in (0,1)$, $\varphi\in {\rm Psh}(\varepsilon B_{x_0})$, there is a holomorphic $n$-form $f_{\varepsilon}$ on $\varepsilon B_{x_0}$ such that $f_{\varepsilon}(x_0)=f_{x_0}$ and
   $$
   \varepsilon^{-2n}\int_{\varepsilon B_{x_0}}\sqrt{-1}^{\,n^2} f_{\varepsilon}\wedge\bar f_{\varepsilon}\, e^{-\varphi}
   \leq {\rm const}_{n,\, \Omega,\, {\rm d}V_X} \frac{\sqrt{-1}^{\,n^2}f_{x_0}\wedge \bar{f}_{x_0}}{{\rm d}V_X(x_0)}e^{-\varphi(x_0)}.
   $$
\end{lemma}

\begin{proof}
Write $f_{x_0}=c\, {\rm d}z_{x_0}$ with $c\in\mathbb C$. As e.g. in \cite{CWW}, an application of the Ohsawa--Takegoshi extension theorem shows that for every $\varepsilon\in (0,1)$ there is a holomorphic $n$-form $f_{\varepsilon}$ on $\varepsilon B_{x_0}$ such that $f_{\varepsilon}(x_0)=f_{x_0}$ and
   $$
   \varepsilon^{-2n}\int_{\varepsilon B_{x_0}}\sqrt{-1}^{\,n^2} f_{\varepsilon}\wedge\bar f_{\varepsilon}\, e^{-\varphi}
   \leq {\rm const}_n |c|^2 e^{-\varphi(x_0)}.
   $$
Combining this with inequality \eqref{ineq:vol-comp} yields the desired result immediately.
\end{proof}

We now proceed to prove Theorem \ref{thm:OT-extension-twisted-restated}. With Theorem \ref{thm:L^2 estimate of OT-type} in hand, we do this by adapting the argument from the proof of \cite[Theorem 1.3]{CWW} to the present context.

Without loss of generality we may assume that ${\rm d}V_X=\omega^n/n!$, with $\omega$  a K\"ahler form on $X$. Fix an arbitrary point $x_0\in \Omega$. Let us write $(B, z)=(B_{x_0}, z_{x_0})$ for the sake of simplicity, and choose a function $\rho\in C^{\infty}(X,\,\mathbb R)$ that is positive, strictly log-psh on $X\!\setminus\!\{x_0\}$ and satisfies $\rho<e^{-1}$ on $\overline{\Omega}$ and $\rho=|z|^2$ on a small neighborhood of $x_0$. Similarly as in \cite{Ohsawa95, Chen_simple, CWW} we set
  $$
  \psi_{\varepsilon}=\log (\rho+\varepsilon^2)\quad {\rm and} \quad \eta_{\varepsilon}=-\psi_{\varepsilon}+\log(-\psi_{\varepsilon})
  $$
for $0<\varepsilon\leq \varepsilon_0$, where $\varepsilon_0>0$ is chosen small enough that $\psi_{\varepsilon_0}<-1$ on $\overline{\Omega}$ and the line bundle $L$ is trivial over $\varepsilon_0B$. A straightforward calculation shows
  $$
  -\partial\eta_{\varepsilon}=(1-\psi_{\varepsilon}^{-1})\partial \psi_{\varepsilon}
  \quad {\rm and} \quad
  -\partial\bar{\partial}\eta_{\varepsilon}=(1-\psi_{\varepsilon}^{-1})\partial\bar{\partial}\psi_{\varepsilon}
  +\psi_{\varepsilon}^{-2}\partial\psi_{\varepsilon} \wedge \bar{\partial}\psi_{\varepsilon},
  $$
whence
  $$
  -\sqrt{-1}\partial\bar{\partial}\eta_{\varepsilon}-\lambda_{\varepsilon}^{-1}\sqrt{-1}\partial\eta_{\varepsilon}\wedge \bar{\partial}\eta_{\varepsilon}\ge \sqrt{-1}\partial\bar{\partial}\psi_{\varepsilon} \quad {\rm on}\ \, \Omega,
  $$
where $\lambda_{\varepsilon}=(1-\psi_{\varepsilon})^2$.

Now let $s_{x_0}\in \Lambda^n T^{\ast}_{x_0}\Omega\otimes L_{x_0}$ be given such that
   $$
   \frac{\sqrt{-1}^{\,n^2}\{s_{x_0}, s_{x_0}\}_h}{\omega^n(x_0)}<\infty.
   $$
Take a holomorphic frame $\zeta$ for $L$ over $\varepsilon_0B$ and write $s_{x_0}$ as $s_{x_0}=f_{x_0}\otimes \zeta_{x_0}$. Let $\{f_{\varepsilon}\}_{0<\varepsilon\leq \varepsilon_0}$ be a family of holomorphic $n$-forms as in Lemma \ref{lem:scaling OT}, corresponding to the weight function
  $$
  \varphi=-\log|\zeta|^2_h,
  $$
which is psh on $\varepsilon_0B$ because of the semi-positivity of $\sqrt {-1}\Theta_h(L)$. Consider
  $$
  \Theta_{\varepsilon}=\sqrt{-1}\partial\bar{\partial}(\rho+\psi_{\varepsilon})
  \quad {\rm and} \quad
  v_{\varepsilon}=\bar{\partial}\big(\chi(|z|^2/\varepsilon^2)f_{\varepsilon}\big)\otimes \zeta,
  $$
where $\chi$ is a $C^{\infty}$ function on $\mathbb R$ satisfying $\chi|_{(-\infty,\,1/4)}=1$, $\chi|_{(3/4,\,\infty)}=0$ and $|\chi'|\leq 4$. Observe that
  $$
  \Theta_{\varepsilon}\geq \sqrt{-1}\partial\bar{\partial}\psi_{\varepsilon}
  =\sqrt{-1}\partial\bar{\partial}\log (|z|^2+\varepsilon^2)
  \geq \frac{\varepsilon^2}{(|z|^2+\varepsilon^2)^2}\sqrt{-1}\partial\bar{\partial}|z|^2
  $$
on a neighborhood of $x_0$. Using this and \cite[Lemme 3.2]{Demailly82} (see also the paragraph following Proposition 1.5 in \cite{Ohsawa95}), we can estimate
\begin{align*}
   \big\langle\negmedspace\big\langle
        [\Theta&_{\varepsilon}, \mathit{\Lambda}_{\omega}]^{-1} v_{\varepsilon},\, v_{\varepsilon} \big\rangle\negmedspace\big\rangle_{\omega,\, he^{-(\rho+n\log\rho)}}\\
   &=\int_{{\rm supp}\,v_{\varepsilon}} \big\langle
        [\Theta_{\varepsilon}, \mathit{\Lambda}_{\sqrt{-1} \partial\bar{\partial}|z|^2}]^{-1}v_{\varepsilon},\,
        v_{\varepsilon} \big\rangle_{\sqrt{-1}\partial\bar{\partial}|z|^2,\,h}
        \, e^{-(\rho+n\log\rho)} \frac{(\sqrt{-1}\partial\bar{\partial}|z|^2)^n}{n!}\\
   &\leq \int_{{\rm supp}\,v_{\varepsilon}} \frac{(|z|^2+\varepsilon^2)^2}{\varepsilon^2}
        |v_{\varepsilon}|^2_{\sqrt{-1}\partial\bar{\partial}|z|^2,\,h}
        \, e^{-(\rho+n\log\rho)} \frac{(\sqrt{-1}\partial\bar{\partial}|z|^2)^n}{n!}\\
   &\leq \|\chi'\|^2_{L^{\infty}(\mathbb R)} \int_{\frac{\sqrt3\varepsilon}{2}B \setminus
        \frac{\varepsilon}{2}B} \frac{|z|^2}{\varepsilon^4}
        \frac{(|z|^2+\varepsilon^2)^2}{\varepsilon^2} \sqrt{-1}^{\,n^2} f_{\varepsilon}\wedge\bar f_{\varepsilon}\, |z|^{-2n} e^{-\varphi-|z|^2}\\
   &\leq {\rm const}_n \, \varepsilon^{-2n} \int_{\varepsilon B}
        \sqrt{-1}^{\,n^2} f_{\varepsilon}\wedge\bar f_{\varepsilon}\, e^{-\varphi}\\
   &\leq {\rm const}_{n,\, \Omega,\, \omega}
        \frac{\sqrt{-1}^{\,n^2} f_{x_0}\wedge\bar{f}_{x_0}}{\omega^n(x_0)} e^{-\varphi(x_0)}\\
   &=   {\rm const}_{n,\, \Omega,\, \omega} \frac{\sqrt{-1}^{\,n^2}\{s_{x_0}, s_{x_0}\}_h}{\omega^n(x_0)}
\end{align*}
for all sufficiently small $\varepsilon>0$. In this way,
  $
  (\Omega, \omega, L, he^{-(\rho+n\log\rho)}, \eta_{\varepsilon}, \lambda_{\varepsilon}, \Theta_{\varepsilon}, v_{\varepsilon})
  $
fits into the framework of Theorem \ref{thm:L^2 estimate of OT-type}.
Therefore,  there is a solution $u_{\varepsilon}$ to equation $\bar{\partial}u=v_{\varepsilon}$ with the estimate
\begin{equation}\label{ineq: solution-estimate1}
  \int_{\Omega} \frac{\sqrt{-1}^{\,n^2}\{u_{\varepsilon}, u_{\varepsilon}\}_h}{\eta_{\varepsilon}+\lambda_{\varepsilon}}
      e^{-(\rho+n\log\rho)}
  \leq {\rm const}_{n,\, \Omega,\, \omega} \frac{\sqrt{-1}^{\,n^2}\{s_{x_0}, s_{x_0}\}_h}{\omega^n(x_0)}.
\end{equation}
Here we have used the identity in \eqref{eq:intr-formula}.
Define
  $$
  s_{\varepsilon}\coloneqq\chi(|z|^2/\varepsilon^2)f_{\varepsilon}\otimes \zeta-u_{\varepsilon}.
  $$
Then $s_{\varepsilon}$ is an $L$-valued holomorphic $n$-form on $\Omega$. Moreover, since  $u_{\varepsilon}$ is holomorphic on a small neighborhood of $x_0$ and $e^{-n\log\rho}$ is not integrable near $x_0$ (as it is equal to $|z|^{-2n}$ there), the finiteness of the left-hand side of \eqref{ineq: solution-estimate1} forces $u_{\varepsilon}(x_0)=0$, and hence
  $$
  s_{\varepsilon}(x_0)=f_{x_0}\otimes \zeta_{x_0}=s_{x_0}.
  $$
Also recalling that $\rho<e^{-1}$ on $\overline{\Omega}$ and using the estimate
  $$
  \eta_{\varepsilon}+\lambda_{\varepsilon}\leq 6\, \psi_{\varepsilon}^2\leq 6 \log^2\!\rho,
  $$
we conclude that
  $$
  (\eta_{\varepsilon}+\lambda_{\varepsilon})^{-1}e^{-(\rho+n\log\rho)}
  \geq \frac{{\rm const}}{\rho^n\log^2\!\rho}\geq {\rm const}_n.
  $$
Hence \eqref{ineq: solution-estimate1} implies
  $$
  \int_{\Omega} \sqrt{-1}^{\,n^2}\{u_{\varepsilon}, u_{\varepsilon}\}_h
  \leq {\rm const}_{n,\, \Omega,\, \omega} \frac{\sqrt{-1}^{\,n^2}\{s_{x_0}, s_{x_0}\}_h}{\omega^n(x_0)}.
  $$
Now combining this with the choice of $f_{\varepsilon}$ and applying the Cauchy--Schwarz inequality, we end up with
  $$
  \int_{\Omega} \sqrt{-1}^{\,n^2}\{s_{\varepsilon}, s_{\varepsilon}\}_h\leq {\rm const}_{n,\, \Omega,\, \omega} \frac{\sqrt{-1}^{\,n^2}\{s_{x_0}, s_{x_0}\}_h}{\omega^n(x_0)}.
  $$
Consequently $s=s_{\varepsilon}$ is a holomorphic section of $\Lambda^n T^{\ast}\Omega \otimes L$
as desired in the theorem, provided $\varepsilon>0$ is sufficiently small. This completes the proof of Theorem \ref{thm:OT-extension-twisted-restated}.

\medskip

Applying Theorem \ref{thm:OT-extension-twisted-restated} to the trivial line bundle $L=\Omega\times
\mathbb C$ equipped with the metric $e^{-\varphi}$, $\varphi\in {\rm Psh}(\Omega)$, we get immediately the following

\begin{theorem}\label{thm:OT-extension}
Let $X$ be an $n$-dimensional Stein manifold with a continuous volume form ${\rm d}V_X$ and $\Omega$ be a relatively compact complete K\"ahler domain in $X$.
Then for every $x_0\in \Omega$, $f_{x_0}\in \Lambda^n T^{\ast}_{x_0}\Omega$ and $\varphi\in {\rm Psh}(\Omega)$,  there exists a holomorphic $n$-form $f$ on $\Omega$ such that $f(x_0)=f_{x_0}$ and
\begin{equation}\label{ineq: extension-ineq}
\int_{\Omega}\sqrt{-1}^{\,n^2} f\wedge\bar f\, e^{-\varphi} \leq {\rm const}_{n,\, \Omega,\, {\rm d}V_X}\frac{\sqrt{-1}^{\,n^2}f_{x_0}\wedge \bar{f}_{x_0}}{{\rm d}V_X(x_0)}e^{-\varphi(x_0)}.
\end{equation}
\end{theorem}

\begin{remark}
Note that the conclusion of Theorem \ref{thm:OT-extension} holds trivially when the right-hand side of \eqref{ineq: extension-ineq} is infinite: in this case every holomorphic $n$-form $f$ on $\Omega$ that takes the value $f_{x_0}$ at $x_0$ automatically satisfies the desired estimate \eqref{ineq: extension-ineq}, and such an $f$ always exists (even on the whole $X$) since $X$ is Stein.

On the other hand, as indicated by a result of Chen, Wu and Wang \cite{CWW}, the conclusion of Theorem \ref{thm:OT-extension} (and hence Theorem \ref{thm:OT-extension-twisted-restated}) does not hold in general when the assumption that $\Omega$ is complete K\"{a}hler is dropped.
\end{remark}

\medskip

\section{Proof of Theorem $\ref{thm:Hartogs_PSH-manifold}$}\label{sect: psh extension}
The key ingredients in the proof are Theorem \ref{thm:OT-extension} and some basic geometric measure theory. As already mentioned in the Introduction, the idea of using Theorem \ref{thm:OT-extension} in the proof was inspired by the work of Chen, Wu and Wang \cite{CWW} (and Demailly \cite{Demailly92}). To get the idea off the ground  we first need to prove the following result, which is also of independent interest.

\begin{theorem}\label{thm:complete-Kahler}
Let $X$ be a Stein manifold and $E\subset X$ be a closed complete pluripolar set. Then $X\!\setminus\!E$ admits a complete K\"{a}hler metric.
\end{theorem}

When $E\subset X$ is a complex subvariety, this result is a  well-known theorem of Grauert \cite{Grauert_comp-kahler} (see also Demailly \cite{Demailly82} for closely related results). Note that even in this special case, $X\!\setminus\!E$ cannot be Stein unless $E$ is empty or has pure codimension one, in view of the second Riemann extension theorem.

\begin{proof}
The result is essentially an easy consequence of Theorem \ref{thm:def-complete-polar}, as we now explain below.

Let $\varphi$ be a $C^{\infty}$ strictly psh exhaustion function for $X$ and let
$\psi\!: X\!\setminus\!E\to \mathbb R$ be a $C^{\infty}$ function such that
   $$
   \psi=-\log(-\rho) \quad \mbox{on}\;\; U\!\setminus\!E,
   $$
where $\rho$ is a psh function as in Theorem \ref{thm:def-complete-polar} with $\sup_X\rho>-1$ and $U\coloneqq\{\rho<-1\}$. Since $-\log(-t)$ is a convex increasing function of $t<0$, $\psi$ is psh on $U\!\setminus\!E$, and hence
$\sqrt{-1}\partial\bar{\partial}\psi\geq 0$ there. One can therefore construct a convex, rapidly increasing $C^{\infty}$ function $\chi$ on $\mathbb R$ such that
   $$
   \omega\coloneqq\sqrt{-1}\partial\bar{\partial}(\chi\circ\varphi)+\sqrt{-1}\partial\bar{\partial}\psi\geq \omega_0 \quad \mbox{on}\;\; X\!\setminus\!E
   $$
for some complete K\"{a}hler metric $\omega_0$  on $X$.

We claim that such an $\omega$ is complete on $X\!\setminus\!E$. To prove this, we may assume without loss of generality that $X$ itself is connected (and so is $X\!\setminus\!E$). Suppose $\{x_j\}_{j\in \mathbb N}$ is a bounded sequence in the metric space $(X\!\setminus\!E,\, \omega)$. Then there is a sequence of smooth curves $\{\gamma_j\}_{j\in \mathbb N}\subset C^{\infty}([0, 1],\, X\!\setminus\!E)$ with uniformly bounded lengths with respect to $\omega$, such that each $\gamma_j$ joins $x_j$ to a fixed reference point in $X\!\setminus\!\overline{U}$. Since $\omega\geq\omega_0$ on $X\!\setminus\!E$ and $\omega_0$ is complete on $X$, we may assume that the sequence $\{x_j\}_{j\in \mathbb N}$ itself converges in $X$  by passing to a subsequence if necessary. What now remains is to show that the limit of $\{x_j\}_{j\in \mathbb N}$ lies outside $E$. Suppose the contrary and set
   $$
   t_j\coloneqq\inf\!\big\{t\in [0,1]\!: \gamma_j([t, 1])\subset U\big\},\quad j\in \mathbb N.
   $$
Clearly $0<t_j<1$ and $\gamma_j(t_j)\in\partial U=\{\rho=-1\}$ for all sufficiently large $j$. Observe also that
   $$
   \omega\geq \sqrt{-1}\partial\bar{\partial}\big(-\log(-\rho)\big)\geq \sqrt{-1}\partial\log(-\rho)\wedge \bar{\partial}\log(-\rho)\quad \mbox{on}\;\; U\!\setminus\!E.
   $$
It then follows that
   \begin{equation*}
      \begin{split}
      \sqrt2\,{\rm length}_{\omega}(\gamma_j)&\geq\int_{t_j}^1\big|\big({\rm d}\log(-\rho)\big)(\gamma_j'(t))\big|\,{\rm d}t \geq\int_{t_j}^1\big({\rm d}\log(-\rho)\big)(\gamma_j'(t))\,{\rm d}t\\
      &=\log(-\rho(x_j))\to \infty \quad {\rm as}\;\;  j\to \infty,
      \end{split}
   \end{equation*}
contradicting the boundedness of $\{{\rm length}_{\omega}(\gamma_j)\}_{j\in \mathbb N}$. Therefore the limit of $\{x_j\}_{j\in \mathbb N}$ lies outside $E$, and hence $\omega$ is complete on $X\!\setminus\!E$.
\end{proof}

\begin{remark}
If we denote by $d_{\omega}$ and $d_{\omega_0}$ the distance functions on $X\!\setminus\!E$ and $X$ induced by  $\omega$ and $\omega_0$, respectively, and by $\chi_{U\setminus E}$  the characteristic function of $U\!\setminus\!E$, then the preceding argument actually yields the estimate
  $$
  d_{\omega}(x, x_0)\geq \max\Big\{d_{\omega_0}(x,x_0),\, \chi_{U\setminus E}(x)\log(-\rho(x))/\sqrt{2}\Big\},
  \quad x\in X\!\setminus\!E,\, x_0\in X\!\setminus\!\overline{U}.
  $$
This, of course, implies that $\omega$ is complete.
\end{remark}

We are now ready to prove Theorem $\ref{thm:Hartogs_PSH-manifold}$.

\begin{proof}[Proof of Theorem $\ref{thm:Hartogs_PSH-manifold}$]

The uniqueness part of the theorem is clear, since two psh functions on $\Omega$ that coincide almost everywhere are actually equal everywhere. So it suffices to prove the existence part.

Let us first reduce the problem to the case when $\Omega\supset K$ is a relatively compact Stein domain in the ambient manifold, which we denote by $X$. To this end, let $\rho$ be a psh function on $X$, continuous outside $K$ and satisfying $\rho^{-1}(-\infty)=K$ (see Theorem \ref{thm:def-complete-polar}). Choose an open set $U\subset X$ such that $K\subset U\subset\subset X$, and set
  $$
  \widetilde{\rho}\coloneqq
  \left\{
  \begin{array}{lll}
  \!\!\max\big\{\rho,\, \inf\limits_{\partial U}\rho\big\} \quad \mbox{on}\;\; X\!\setminus\!\overline{U}; \\
  \!\!\qquad  \quad  \rho\quad  \quad  \quad \;\;\,\mbox{on}\;\;\overline{U}.
  \end{array}
  \right.
  $$
Then  $\widetilde{\rho}$ is a psh function on $X$ with $\widetilde{\rho}^{\,-1}(-\infty)=K$. Replacing $\Omega$ by any connected component of $\{\widetilde{\rho}<\inf_{\partial U}\rho\}$, we may assume in what follows that $\Omega$ is a relatively compact Stein domain in $X$.

Now $\Omega\!\setminus\!K$ is a relatively compact complete K\"{a}hler domain in $X$, in view of Theorem \ref{thm:complete-Kahler}. Write $n={\rm dim}_{\mathbb C} X$ and fix
a continuous volume form ${\rm d}V_X$ on $X$. Given $\varphi\in {\rm Psh}(\Omega\!\setminus\!K)$, we can therefore invoke Theorem \ref{thm:OT-extension} to assign to every point $x\in\Omega\!\setminus\!K$ a holomorphic $n$-form $f_x$ on $\Omega\!\setminus\!K$ with the property that
\begin{equation}\label{ineq:norm-estimate}
  \frac{\sqrt{-1}^{\,n^2}f_x\wedge\bar{f}_x}{{\rm d}V_X}(x)=e^{\varphi(x)}
  \quad {\rm and} \quad \int_{\Omega\setminus K}\sqrt{-1}^{\,n^2} f_x\wedge\bar f_x\, e^{-\varphi} \leq {\rm const},
\end{equation}
where the constant involved is independent of $x\in\Omega\!\setminus\!K$. By the Hartogs extension theorem for holomorphic forms (see, e.g., \cite[p. 22, Corollaire du Th\'{e}or\`{e}me 3]{Serre_duality}), each such $f_x$ extends holomorphically to $\Omega$. With a slight abuse of notation, we still denote this extension by $f_x$.

We next prove the psh extendibility of $\varphi$ across $K$. By a classical result of Lelong, it suffices to show that every point of $K$ admits a small neighborhood on which $\varphi$ is bounded above. Fix an arbitrary point $x_0\in K$. Since the problem is local, we may assume without loss of generality that $\Omega$ is a bounded pseudoconvex domain in $\mathbb C^n$ containing $x_0$. Let $z=(z_1,\ldots, z_n)$ be the coordinate of $\Omega$. Identifying $x$ with its coordinate $z=z(x)$ and $f_x$ with $f_x/{\rm d}z_1\wedge\cdots\wedge {\rm d}z_n$, we conclude from \eqref{ineq:norm-estimate}, after shrinking $\Omega$ if necessary, that
\begin{equation}\label{ineq:norm-local-estimate}
  e^{\varphi(z)}\leq {\rm const}\, |f_z(z)|^2 \quad {\rm and} \quad
  \int_{\Omega\setminus K}|f_z|^2e^{-\varphi}\leq {\rm const},
\end{equation}
where the integral involved is taken against the Lebesgue measure on $\mathbb C^n$. To proceed with the proof we now recall that, as a polar subset of $\mathbb C^n\cong \mathbb R^{2n}$, $K$ has Hausdorff dimension at most $2n-2$ (see, e.g., \cite[Theorem 5.9.6]{ArGa_potential}). This enables us to make use of a result of Shiffman (see \cite[Lemma 2.3]{Harvey_Survey} or  \cite[Chapter III, Lemma 4.7]{Demaillybook}), according to which we can find---by suitably selecting affine linear coordinates for $\mathbb C^n$---a polydisc $\Delta'\times\Delta''\subset \mathbb C^{n-1}\times\mathbb C$ centered at $z_0=z(x_0)\eqqcolon(z'_0, z''_0)$ such that
  $$
  (\Delta'\times\partial\Delta'')\cap K=\emptyset.
  $$
By shrinking $\Delta'$ if necessary, we further arrive at
  $$
  \big(\Delta'\times (\Delta''\!\setminus\!(1-\varepsilon)\overline{\Delta''})\big)\cap K=\emptyset
  $$
for some sufficiently small $\varepsilon>0$. The remaining argument is the same as in the last paragraph of the proof of \cite[Theorem 1.2]{CWW}. Choose $R,\, r>0$ such that
  $$
  1-\varepsilon<r<R<1
  $$
and a ball $B\subset\subset\Delta'$ centered at $z'_0$. Then the Cauchy estimate for holomorphic functions and the inequalities in \eqref{ineq:norm-local-estimate} imply
\begin{equation*}
\begin{split}
e^{\varphi(z)}&\leq {\rm const}\, |f_z(z)|^2
\leq {\rm const}_{B,\, R,\, r,\, \varepsilon}\int_{B\times (R\Delta''\setminus r\Delta'')}|f_z|^2\\
&\leq {\rm const}_{B,\, R,\, r,\, \varepsilon}\sup_{B\times (R\Delta''\setminus r\Delta'')}e^\varphi
      \int_{B\times (R\Delta''\setminus r\Delta'')}|f_z|^2e^{-\varphi}\\
&\leq C\sup_{B\times (R\Delta''\setminus r\Delta'')}e^\varphi
\end{split}
\end{equation*}
for all $z\in \big((1-\varepsilon)(B\times\Delta'')\big)\setminus K$, where $C>0$ is a constant independent of $z$. Consequently, $\varphi$ is bounded above on $(1-\varepsilon)(B\times\Delta'')\ni z_0$ outside $K$. This completes the proof.
\end{proof}

\end{document}